\documentclass[aop,MSNbibl,number,citesort,seceqn,dvips]{arximspdf}
\usepackage{mathbh}

\doi{10.1214/11-AOP699}
\volume{41}
\issue{2}
\pubyear{2013}
\firstpage{722}
\lastpage{743}

\makeatletter
\newtheorem{theorem}{Theorem}
\newtheorem{lemma}{Lemma}
\newtheorem{proposition}{Proposition}
\newproclaim{remark}{Remark}

\def\W{\Omega}
\def\demi{{1\over2}}
\def\Z{{\mathbb{Z}}}
\def\P{{\mathbb{P}}}
\def\N{{\mathbb{N}}}
\def\R{{\mathbb{R}}}
\def\ue{\underline{e}}
\def\oe{\overline{e}}
\def\E{{\mathbb{E}}}
\def\dive{\operatorname{div}}
\def\hhh{{\mathcal{H}}}
\def\indic{\mathbh{1}}
\def\Q{\mathbb{Q}}
\makeatother

\begin{document}
\begin{frontmatter}

\title{Random Dirichlet environment viewed from the particle in
dimension $d\ge3$\thanksref{T1}}
\runtitle{RWDE viewed from the particle}

\thankstext{T1}{Supported by the ANR project MEMEMO.}

\begin{aug}
\author[A]{\fnms{Christophe} \snm{Sabot}\corref{}\ead[label=e1]{sabot@math.univ-lyon1.fr}}
\runauthor{C. Sabot}
\affiliation{Universit\'{e} de Lyon}
\address[A]{Universit\'{e} de Lyon\\
Universit\'{e} Lyon 1, CNRS UMR5208\\
Institut Camille Jordan, 43, bd du 11 nov.,\\
69622 Villeurbanne cedex\\
France\\
\printead{e1}} 
\end{aug}

\received{\smonth{7} \syear{2010}}
\revised{\smonth{7} \syear{2011}}

%
\begin{abstract}
We consider random walks in random
Dirichlet environment\break (RWDE), which is a special type of random
walks in random environment where the exit probabilities at each
site are i.i.d. Dirichlet random variables. On $\Z^d$, RWDE are
parameterized by a $2d$-tuple of positive reals called weights. In
this paper, we characterize for $d\ge3$ the weights for which
there exists an absolutely continuous invariant probability
distribution for
the process viewed from the particle. We can deduce from this
result and from [\textit{Ann. Inst. Henri Poincar\'e Probab. Stat.}
\textbf{47} (2011) 1--8] a complete description of
the ballistic regime for $d\ge3$.
\end{abstract}

%

\end{frontmatter}
%

\section{Introduction}\label{sec1}
Multidimensional random walks in random environment have received
a considerable attention in the last ten years. Some important
progress has been made in the ballistic regime (after the seminal works
\cite{Kalikow,Sznitman-Zerner,Sznitman1,Sznitman2})
and for small perturbations of the simple random walk \cite
{Sznitman-Zeitouni,Bolthausen-Zeitouni}.
We refer to~\cite{Zeitouni} for a detailed survey. Nevertheless, we are
still far
from a complete description, and some basic questions are
open such as the characterization of recurrence, ballisticity. The
point of view of the environment viewed from the particle has
been a powerful tool to investigate the random conductance model; it is
a key ingredient in
the proof of invariance principles \cite
{Kipnis-Varadhan,Kozlov,Sidoravicius-Sznitman,Mathieu}
but has had a rather little impact on the nonreversible model. The
existence of an absolutely
continuous invariant measure for the process viewed from the
particle (the so called ``equivalence of the static and dynamical point
of view'') is only known in a few cases:
for dimension 1, cf. Kesten~\cite{Kesten} and Molchanov \cite
{Molchanov} pages 273--274; in the case of balanced
environment of Lawler~\cite{Lawler}; for ``nonnestling'' RWRE in
dimension $d\ge4$ at low disorder, cf. Bolthausen and
Sznitman~\cite{Bolthausen-Sznitman}; and in a weaker form for ballistic
RWRE (equivalence in
half-space), cf.~\cite{Rassoul-Agha,Rassoul-Agha-Seppalainen}.
Note that invariance principles have nevertheless been obtained under
special assumptions:
under the ballistic assumption~\cite{Rassoul-Agha-Seppalainen,Berger-Zeitouni}
and for weak disorder in dimension $d\ge3$,
\cite{Sznitman-Zeitouni,Bricmond-Kupianen}.


Random walks in Dirichlet environment (RWDE) is a special case
where at each site the environment is chosen according to a
Dirichlet random variable. One remarkable property of Dirichlet
environments is that the annealed law
of RWDE is the law of a directed edge reinforced random walk as
remarked initially in Pemantle's Ph.D. thesis
\cite{Pemantle1,Pemantle2}, the idea of reinforced random
walks going back to Diaconis and Coppersmith; cf. \cite
{Pemantle-survey} for a
survey. While this
model of environment is fully random (the support of the distribution
on the environment is the space of weakly
elliptic environment itself), it
shows some surprising analytic simplifications; cf.
\cite{Sabot,hypergeom,Sabot-Tournier,Enriquez-Sabot,Enriquez-Sabot-Zindy}. In
particular, in~\cite{Sabot}, the author proved that RWDE are transient on
transient graphs; cf.~\cite{Sabot} for a precise result. This
result uses in a crucial way a property of statistical invariance
by time reversal; cf. Lemma 1 of~\cite{Sabot}.

RWDE are parametrized by $2d$ reals called the weights (one for each
direction in $\Z^d$) which govern the behavior of the walk.
In this paper we characterize on $\Z^d$, $d\ge3$, the weights for
which there exists an
invariant probability measure for the environment viewed from the
particle, which is absolutely continuous with
respect to the law of the environment. More precisely, it is shown that
there is an absolutely
continuous invariant probability exactly when the parameters are such
that the time spent in finite size traps has finite expectation.
Together with previous results on directional transience \cite
{Sabot-Tournier} it leads,
using classical results on stationary ergodic sequences, to a complete
description of
the ballistic regimes for RWDE in dimension larger or equal to 3.
Besides, we think that the proof of the existence of an absolutely
continuous invariant distribution for the environment viewed from the
particle could be a first step toward an implementation
of the technics developed to prove functional central limit theorems;
cf., for example,~\cite{Komorowski-Olla}.

\section{Statement of the results}\label{sec2}
Let $(e_1, \ldots,e_d)$ be the canonical base of $\Z^d$, and set
$e_j=-e_{j-d}$, for $j=d+1, \ldots,2d$. The set $\{e_1, \ldots
,e_{2d}\}$ is the set of unit vectors of $\Z^d$. We denote by
$\|z\|=\sum_{i=1}^d \vert z_i\vert$ the $L_1$-norm of $z\in\Z^d$.
We write $x\sim y$ if $\|y-x\|=1$. We consider elliptic random
walks in random environment to nearest neighbors. We denote by
$\W$ the set of environments
\begin{eqnarray*}
&&\W=\Biggl\{\omega =(\omega (x,y))_{x\sim y}\in\,]0,1]^E,\\
&&\phantom{\W=\Biggl\{} \mbox{such that
for all
$x\in\Z^d$, } \sum_{i=1}^{2d} \omega (x,x+e_i)=1\Biggr\}.
\end{eqnarray*}
An environment $\omega $ defines the transition probability of a Markov
chain on $\Z^d$, and we denote by $P^{\omega }_x$ the law of this
Markov chain starting from $x$.
\[
P^{\omega }_x[ X_{n+1}=y+e_i| X_n=y]=\omega (y,y+e_i).
\]

The classical model of nonreversible random environment
corresponds to the model where at each site $x\in\Z^d$ the
environment $(\omega (x,x+e_i))_{i=1, \ldots, 2d}$ is chosen
independently according to the same law. Random Dirichlet
environment corresponds to the case where this law is a Dirichlet
law. More precisely, we choose some positive weights $(\alpha_1,
\ldots,\alpha_{2d})$, and we define $\lambda=\lambda^{(\alpha)}$
as the Dirichlet law with parameters $(\alpha_1, \ldots,
\alpha_{2d})$. It means that $\lambda^{(\alpha)}$ is the law on
the simplex
%
\begin{equation}
\label{T2d}
\Biggl\{ (x_1, \ldots,x_{2d})\in\,]0,1]^{2d}, \sum_{i=1}^{2d} x_i=1\Biggr\}
\end{equation}
with density
%
\begin{equation} \label{simplex}
{\Gamma(\sum_{i=1}^{2d} \alpha_i)\over\prod_{i=1}^{2d}
\Gamma(\alpha_i)} \Biggl( \prod_{i=1}^{2d} x_i^{\alpha_i-1} \Biggr)
\,dx_1\cdots \,dx_{2d-1},
\end{equation}
where $\Gamma$ is the usual Gamma function
$\Gamma(\alpha)=\int_0^\infty t^{\alpha-1} e^{-t} dt$. [In the
previous expression $dx_1\cdots \,dx_{2d-1}$ represents the image of
the Lebesgue measure on $\R^{2d-1}$ by the application $(x_1,
\ldots, x_{2d-1})\rightarrow(x_1, \ldots, x_{2d-1},
1-(x_1+\cdots+x_{2d-1})]$. Obviously, the law does not depend on
the specific role of $x_{2d}$.) We denote by $\P^{(\alpha)}$ the
law obtained on $\W$ by picking at each site $x\in\Z^d$ the
transition probabilities $(\omega (x,x+e_i))_{i=1, \ldots, 2d}$
independently according to $\lambda^{(\alpha)}$. We denote by $\E
^{(\alpha)}$
the expectation with respect to $\P^{(\alpha)}$ and by
$\P^{(\alpha)}_x[\cdot]=\E^{(\alpha)}[ P^{(\omega )}_x(\cdot)]$
the annealed
law of the process starting at $x$.
This type of environment plays a special role since the annealed law
corresponds to a directed edge reinforced random walk with an
affine reinforcement, that is,
\[
\P^{(\alpha)}_x [X_{n+1}=X_n+e_i | \sigma(X_k, k\le n)]=
{\alpha_i+N_i(X_n, n)\over\sum_{k=1}^{2d} \alpha_k+N_k(X_n, n)},
\]
where $N_k(x,n)$ is the number of crossings of the directed edge
$(x,x+e_k)$ up to time $n$. This is just a consequence of the fact that
the Dirichlet law is the mixing measure
of Polya urns so that at each site the annealed process choose a
direction following a Polya urn with parameters
$(\alpha_i)_{i=1, \ldots, 2d}$; cf.~\cite{Pemantle1} or~\cite{Pemantle2}.
When the weights are constant equal to $\alpha$, the environment
is isotropic: when $\alpha$ is large, the environment is close to
the deterministic environment of the simple random walk, when
$\alpha$ is small the environment is very disordered. The
following parameter $\kappa$ is important in the description of
the RWDE:
\[
\kappa= 2\Biggl(\sum_{i=1}^{2d} \alpha_i\Biggr) -\max_{i=1, \ldots
, d} (\alpha_{i}+\alpha_{i+d}).
\]
If $i_0\in\{1, \ldots,d\}$ realizes the maximum in the last term, then
$\kappa$ is
the sum of the weights of the edges exiting the set $\{0,e_{i_0}\}$ (or
$\{0,-e_{i_0}\}$).
The real $\kappa$ must be understood as the strength of the trap $\{
0,e_{i_0}\}$: indeed,
if $\tilde G^\omega (0,0)$ is the Green function at $(0,0)$ of the Markov
chain in environment $\omega $ killed
at its exit time of the set $\{0,e_{i_0}\}$, then $\tilde G^\omega (0,0)^s$
is integrable if and only if $s<\kappa$~\cite{Tournier}.
In~\cite{Sabot} it has been proved for $d\ge3$ that the same is true
for the Green
function $G(0,0)$ on $\Z^d$ itself: it has
integrable $s$-moment if and only if $s<\kappa$.

Denote by $(\tau_x)_{x\in\Z^d}$ the shift on the environment
defined by
\[
\tau_x\omega (y,z)=\omega (x+y,x+z).
\]
Let $X_n$ be the random walk in environment $\omega $. The process
viewed from the particle is the process on the state space $\W$
defined by
\[
\overline\omega _n=\tau_{X_n}\omega .
\]
Under $P^{\omega _0}_0$, $\omega _0\in\W$ (resp., under $\P_0$),
$\overline\omega _n$ is a Markov process on state space $\W$ with
generator $R$ given by
\[
Rf(\omega )=\sum_{i=1}^{2d} \omega (0,e_i) f(\tau_{e_i} \omega ),
\]
for all bounded measurable function $f$ on $\W$, and with initial
distribution $\delta_{\omega _0}$ (resp., $\P$); cf., e.g.,
\cite{Sznitman-ten}. Compared to the quenched process, the process
viewed from the particle is Markovian. Since the state space is
huge, one needs to take advantage of this point of view, to have
the existence of an invariant probability measure, absolutely
continuous with respect to the initial measure on the environment.
The following theorem solves this problem in the special case of
Dirichlet environment in dimension $d\ge3$ and is the main result
of the paper.
\begin{theorem}\label{main-result}
Let $d\ge3$ and $\P^{(\alpha)}$ be the law of the Dirichlet
environment with weights $(\alpha_1, \ldots, \alpha_{2d})$. Let
$\kappa>0$ be defined by
\[
\kappa= 2\Biggl(\sum_{i=1}^{2d} \alpha_i\Biggr) -\max_{i=1, \ldots
, d} (\alpha_{i}+\alpha_{i+d}).
\]

\begin{longlist}[(ii)]
\item[(i)] If $\kappa>1$, then there exists a unique probability
distribution $\Q^{(\alpha)}$ on $\W$ absolutely continuous with
respect to $\P^{(\alpha)}$ and invariant by the generator $R$.
Moreover ${d\Q^{(\alpha)}\over d\P^{(\alpha)}}$ is in $L_p(\P
^{(\alpha
)})$ for
all $1\le p<\kappa$.\vspace*{2pt}

\item[(ii)] If $\kappa\le1$, there does not exist any probability
measure invariant by $R$ and absolutely continuous with respect to
the measure $\P^{(\alpha)}$.
\end{longlist}
\end{theorem}

We can deduce from this result and from~\cite{Tournier,Sabot-Tournier}, a characterization of ballisticity for
$d\ge3$. Let $d_\alpha$ be the mean drift at first step
\[
d_\alpha=\E^{(\alpha)}_0(X_1)={1\over\sum_{i=1}^{2d}
\alpha_i}\sum_{i=1}^{2d} \alpha_i e_i.
\]

\begin{theorem}\label{ballisticity}
Let $d\ge3$.
\begin{longlist}[(iii)]
\item[(i)] (cf.~\cite{Tournier}) If $\kappa\le1$, then
\[
\lim_{n\to\infty} {X_n\over n} =0,\qquad \P^{(\alpha)}_0 \mbox{ a.s.}
\]
\item[(ii)] If $\kappa>1$ and $d_\alpha=0$, then
\[
\lim_{n\to\infty} {X_n\over n} =0,\qquad \P^{(\alpha)}_0 \mbox{ a.s.}
\]
and for all $i=1, \ldots, d$
\[
\liminf X_n\cdot e_i=-\infty,\qquad \limsup X_n\cdot e_i=+\infty,\qquad \P
^{(\alpha
)}_0\mbox{ a.s.}
\]
\item[(iii)] If $\kappa>1$ and $d_\alpha\neq0$, then there exists $v\neq
0$ such that
\[
\lim_{n\to\infty} {X_n\over n} =v,\qquad \P^{(\alpha)}_0 \mbox{ a.s.}
\]
Moreover, for the integers $i\in\{1, \ldots, d\}$ such that $d_\alpha
\cdot e_i\neq0$ we have
\[
(d_\alpha\cdot e_i)(v\cdot e_i)>0.
\]
For the integers $i\in\{1, \ldots, d\}$ such that $d_\alpha\cdot e_i=0$
\[
\liminf X_n\cdot e_i=-\infty,\qquad \limsup X_n\cdot e_i=+\infty,\qquad \P
^{(\alpha
)}_0\mbox{ a.s.}
\]
\end{longlist}
\end{theorem}
\begin{remark} This answers in the case of RWDE for $d\ge3$ the
following question: is directional transience equivalent to
ballisticity? The answer is formally ``no'' but morally ``yes'':
indeed, it is proved in~\cite{Sabot-Tournier} that for all $i$ such
that $d_\alpha\cdot e_i\neq0$, $X_n\cdot e_i$ is transient;
hence, for $\kappa\le1$ directional transience and zero speed can
coexist. But, it appears in the proof of~\cite{Tournier} that the
zero speed is due to finite size traps that come from the
nonellipticity of the environment. When $\kappa>1$, the expected
exit time of finite boxes is always finite (cf.~\cite{Tournier})
and in this case (ii) and~(iii) indeed tell that directional
transience is equivalent to ballisticity. For general
RWRE (and for RWDE in dimension 2) this is an important unsolved question.
Partial important results in this
direction have been obtained by Sznitman in
\cite{Sznitman1,Sznitman2} for general uniformly elliptic environment for
$d\ge2$.
\end{remark}
\begin{remark}
A law of of large number (with eventually random or null velocity) has
been proved for general (weakly) elliptic RWRE
by Zerner (cf.~\cite{Zerner}) using the technics of regeneration times
developed by Sznitman and Zerner in\vadjust{\goodbreak}
\cite{Sznitman-Zerner}. Nevertheless, when the directional 0--1 law is
not valid it is still not known whether there is
a deterministic limiting velocity (this was solved for $d\ge5$ by
Berger,~\cite{Berger}).
\end{remark}
\begin{remark}
In case (i), it would be interesting to understand the behavior of
$X_n\cdot e_i$ depending on the value of
$d_\alpha\cdot e_i$ as in (ii) and (iii). It is not yet possible due to
the absence of absolutely continuous invariant measure for the
process viewed from the particle. We nevertheless think that this
question should be settled in a further work.
\end{remark}

\section{\texorpdfstring{Proof of Theorem \protect\ref{main-result}(i)}{Proof of Theorem 1(i)}}\label{sec3}
Let us first recall a few definitions and give some notations. By
a directed graph we mean a pair $G=(V,E)$ where $V$ is a
countable set of vertices and $E$ the set of (directed) edges is a
subset of $V\times V$. For simplicity, we do not allow multiple
edges or loops [i.e., edges of the type $(x,x)$]. We denote by $\ue$,
respectively $\oe$, the tail and the
head of an edge $e\in E$, so that $e=(\ue,\oe)$. A~directed path
from a vertex $x$ to a vertex $y$ is a sequence $\sigma=(x_0=x,
\ldots, x_n=y)$ such that for all $i=1, \ldots, n,$
$(x_{i-1},x_i)$ is in $E$. The divergence operator is the function
$\dive:\R^E\mapsto\R^V$ defined for
$\theta\in\R^E$ by
\[
\forall x\in V,\qquad
\dive(\theta)(x)=\sum_{e\in E, \ue=x} \theta(e)-\sum_{e\in E, \oe=x}
\theta(e).
\]

We consider $\Z^d$ as a directed graph:
$G_{\Z^d}=(\Z^d, E)$ where the edges are the pair $(x,y)$ such
that $\|y-x\|=1$. On $E$ we consider the weights $(\alpha(e))_{e\in E}$
defined by
\[
\forall x\in\Z^d, i=1, \ldots, 2d,\qquad \alpha\bigl((x,x+e_i)\bigr)=\alpha_i.
\]
Hence, under $\P^{(\alpha)}$, at each site $x\in\Z^d$, the exit
probabilities $(\omega (e))_{\ue=x}$
are independent and distributed according to a Dirichlet law with
parameters $(\alpha(e))_{\ue=x}$.

When $N\in\N^*$, we denote by $T_N=(\Z/N\Z)^d$
the $d$-dimensional torus of size $N$. We denote by
$G_N=(T_N,E_N)$ the associated directed graph image of the graph
$G=(\Z^d, E)$ by projection on the torus. We denote by $d(\cdot,
\cdot)$ the shortest path distance on the torus. We write $x\sim
y$ if $(x,y)\in E_N$. Let $\W_N$ be the space of (weakly) elliptic
environments on $T_N$
\begin{eqnarray*}
&&\W_N=\Biggl\{\omega =(\omega (x,y))_{(x,y)\in E_N}\in\,]0,1]^{E_N}, \\
&&\phantom{\W_N=\Biggl\{}\mbox{such that }\forall x\in T_N, \sum_{i=1}^{2d} \omega (x,x+e_i)=1\Biggr\}.
\end{eqnarray*}
$\W_N$ is naturally identified with the
space of the $N$-periodic environments on $\Z^d$.
We denote by $\P_N^{(\alpha)}$ the Dirichlet law on the environment obtained
by picking independently at each site $x\in T_N$ the exiting
probabilities $(\omega (x,x+e_i))_{i=1, \ldots, 2d}$
according to a
Dirichlet law with parameters $(\alpha_i)_{i=1, \ldots, 2d}$.\vadjust{\goodbreak}

For $\omega $ in
$\W_N$ we denote by $(\pi_N^\omega (x))_{x\in T_N}$ the invariant
probability measure of the Markov chain on $T_N$ with transition
probabilities $\omega $ (it is unique since the environments are
elliptic). Let
\[
f_N(\omega )= N^d \pi_N^\omega (0),
\]
and
\[
\Q_N^{(\alpha)}=f_N\cdot\P_N^{(\alpha)}.
\]
Thanks to translation invariance, $\Q_N^{(\alpha)}$ is a probability
measure on
$\W_N$.
Theorem~\ref{main-result} is a consequence of the following
lemma.
\begin{lemma}\label{main-lemma}
Let $d\ge3$. For all $p\in[1,\kappa[$
\[
\sup_{N\in\N} \| f_N\|_{L_p(\P^{(\alpha)}_N)} <\infty.
\]
\end{lemma}

Once this lemma is proved, the proof of Theorem~\ref{main-result}
is routine argument; cf., for example,~\cite{Sznitman-ten}, pages
18 and 19. Indeed, we consider $\P^{(\alpha)}_N$ and $\Q^{(\alpha)}_N$ as
probability
measures on $N$-periodic environments. Obviously, $\P^{(\alpha)}_N$ converges
weakly to the probability measure $\P^{(\alpha)}$. By construction,
$\Q
^{(\alpha)}_N$ is
an invariant probability measure for the process of the
environment viewed from the particle. Since $\W$ is compact, we
can find a subsequence $N_k$ such that $\Q^{(\alpha)}_{N_k}$
converges weakly
to a probability measure $\Q^{(\alpha)}$ on $\W$. The probability
$\Q
^{(\alpha)}$ is
invariant for the process viewed from the particle, as a
consequence of the invariance of $\Q^{(\alpha)}_N$. Let $g$ be a continuous
bounded function on $\W$: we have for $p$ such that $1<p<\kappa$
and $q={p\over p-1}$
\begin{eqnarray*}
\biggl\vert\int g d\Q^{(\alpha)}\biggr\vert&= & \biggl\vert
\lim_{k\to\infty} \int g f_{N_k} d\P^{(\alpha)}_{N_k}\biggr\vert
\\
&\le& \limsup_{k\to\infty}\biggl(\int\vert g\vert^q d\P^{(\alpha)}_{N_k}
\biggr)^{1/q}\biggl(\int f_{N_k}^{p} d\P^{(\alpha)}_{N_k}\biggr)^{1/ p}
\\
&\le& c_p\|g\|_{L_q(\P^{(\alpha)})},
\end{eqnarray*}
where
\[
c_p=\sup_{N\in\N} \| f_N\|_{L_p(\P^{(\alpha)}_N)} <\infty.
\]
As a consequence, $\Q^{(\alpha)}$ is absolutely continuous with
respect to
$\P^{(\alpha)}$ and
\[
\biggl\|{d\Q^{(\alpha)}\over d\P^{(\alpha)}}\biggr\|_{L_p(\P^{(\alpha)})}\le c_p.
\]
The uniqueness of $\Q^{(\alpha)}$ is classical and proved, for example,
in~\cite{Sznitman-ten}, page 11.

\begin{pf*}{Proof of Lemma \protect\ref{main-lemma}}
The proof is divided into three steps. The first step prepares the
application of the property
of ``time reversal invariance'' (Lemma 1 of~\cite{Sabot} or
Proposition 1 of~\cite{Sabot-Tournier}).
The second step is a little trick to increase the weights in order to
get the optimal
exponent. The third step makes a crucial use of the
``time-reversal invariance''
and uses a lemma of the type ``max-flow
min-cut problem'' proved in the next section.

\textit{Step} 1: Let $(\omega _{x,y})_{x\sim y}$ be in $\W_N$. The
time-reversed environment is defined by
\[
\check w_{x,y}=\pi_N^\omega (y) \omega _{y,x}{1\over\pi_N^\omega (x)}
\]
for $x$, $y$ in $T_N$, $x\sim y$. At each point $x\in T_N$
\[
\sum_{\ue=x}\alpha(e)=\sum_{\oe=x}\alpha(e) =\sum_{j=1}^{2d}
\alpha_j.
\]
It implies by Lemma 1 of
\cite{Sabot} that if $(\omega _{x,y})$ is distributed according to
$\P
^{(\alpha)}$,
then~$\check w$ is distributed according to
$\P^{(\check\alpha)}$ where
\[
\forall(x,y)\in E_N, \qquad\check\alpha_{(x,y)}=\alpha_{(y,x)}.
\]
Let $p$ be a real, $1\le p<\kappa$
%
\begin{eqnarray}\label{inequality}
(f_N)^p&=&
(N^{d}\pi_N^\omega (0))^p \nonumber
\\
&=& \biggl({\pi_N^\omega (0)\over{1/N^d}\sum_{y\in T_N} \pi_N^\omega (y)}\biggr)^p
\\
&\le& \prod_{y\in T_N}\biggl ({\pi_N^\omega (0)\over
\pi_N^\omega (y)}\biggr)^{p/N^d},\nonumber
\end{eqnarray}
where in the last inequality we used the arithmetico-geometric
inequality.
If $\theta:E_N\rightarrow\R_+$, we define $\check
\theta$ by
\[
\check\theta_{(x,y)}=\theta_{(y,x)}\qquad \forall x\sim y.
\]
For two functions $\gamma$ and $\beta$ on $E_N$ (resp., on
$T_N$), we write $\gamma^\beta$ for $\prod_{e\in E_N}
\gamma(e)^{\beta(e)}$ (resp., $\prod_{x\in T_N}
\gamma(x)^{\beta(x)}$). We clearly have
%
\begin{eqnarray}\label{quotient}
{\check\omega ^{\check\theta}\over\omega ^{\theta}}&=&
\prod_{e\in E_N} {( \omega (e)\pi_N(\ue)\pi_N(\oe)^{-1})^{\theta
(e)}\over
\omega _e^{\theta(e)}}
\nonumber\\
&=&\prod_{x\in T_N} \pi_N(x)^{\sum_{e, \ue=x}\theta(e)-\sum
_{e,\oe=x}
\theta(e)}
\\
&=& \pi_N^{\dive(\theta)}.\nonumber
\end{eqnarray}
Hence, for all $\theta:E_N\mapsto\R_+$
such that
%
\begin{equation}\label{dive-theta}
\dive(\theta)={p\over N^d}\sum_{y\in
T_N}(\delta_0-\delta_y)
\end{equation}
we have, using (\ref{inequality}) and (\ref{quotient}),
%
\begin{equation}\label{fN}
f_N^p\le
{\check\omega ^{\check\theta}\over\omega ^{\theta}}.
\end{equation}

\textit{Step 2:}
Considering that $1=\sum_{\|e\|=1}\omega (0,e)$, we have
\[\label{1}
1=1^\kappa\le(2d)^\kappa\sum_{i=1}^{2d}\omega (0,e_i)^\kappa.
\]
Hence, we get
\[
\E^{(\alpha)}(f_N^p)\le (2d)^\kappa\sum_{i=1}^{2d}
\E^{(\alpha)}(\omega (0,e_i)^\kappa f_N^p).
\]
Hence, we need now to prove that for all
$i=1, \ldots, 2d,$
%
\begin{equation}
\sup_{N\in\N} \E^{(\alpha)}(\omega (0,e_i)^\kappa f_N^p)<\infty.
\end{equation}
Considering (\ref{fN}), we need to prove that for all
$i=1, \ldots, 2d,$ we can find a sequence $(\theta_N)$, where
$\theta_N:E_N\mapsto\R_+$ satisfies (\ref{dive-theta}) for all
$N$, such that
%
\begin{equation}\label{sup-alpha}
\sup_{N\in\N} \E^{(\alpha)}\biggl(
\omega _{(0,e_i)}^\kappa{\check\omega ^{\check\theta_N}\over
\omega ^{\theta_N}}
\biggr)<\infty.
\end{equation}

\textit{Step 3:} This is related to the max-flow min-cut
problem; cf., for example,~\cite{Lyons-Peres}, Section 3.1 or \cite
{Ford-Fulkerson}.
Let us first recall the notion of
minimal cut-set sums on the graph $G_{\Z^d}$.
A cut-set between $x\in\Z^d$ and $\infty$ is a subset $S$ of $E$
such that
any infinite simple directed path (i.e., an infinite directed path
that does not pass twice by the same vertex) starting from $x$
must pass through one (directed) edge of $S$. A~cut-set which is minimal
for inclusion is necessarily of the form
%
\begin{equation}\label{min-cut-set}
S=\partial_+(K)=\{e\in E, \ue\in K, \oe\in K^c\},
\end{equation}
where $K$ is a finite subset of $\Z^d$ containing $x$ such that
any $y\in K$ can be reached by a directed path in $K$ starting at
$x$. Let $(c(e))_{e\in E}$ be a set of nonnegative reals called
the capacities. The minimal cut-set sum between 0 and $\infty$ is
defined as the value
\[
m((c))=\inf\{c(S), \mbox{$S$ is a cut-set separating $0$ and
$\infty$}\},
\]
where $c(S)=\sum_{e\in S} c(e)$. Observe that the infimum can be
taken only on minimal cut-set, that is, cut-set of the form
(\ref{min-cut-set}).

The proof uses the following lemma, whose proof is deferred to the
next section since it is of a different nature.
\begin{lemma}\label{L2-max-flow}
Let $d\ge3$.
Let $(c(e))_{e\in E}$ be such that
\[
\inf_{e\in E} c(e) >0;\qquad \sup_{e\in E} c(e) <\infty.
\]
There exists a constant $c_1>0$ such that for $N$ large enough
there exists a function $\theta_N:E_N\mapsto\R_+$ such that
%
\begin{eqnarray}\label{flow-equation}
\dive(\theta_N)&=&m((c)){1\over N^d}\sum_{x\in T_N}(\delta_0 -
\delta_x),
\nonumber
\\[-8pt]
\\[-8pt]
\nonumber
\|\theta_N\|^2_2&=&\sum_{e\in E_N}\theta_N(e)^2<c_1
\end{eqnarray}
and such that
%
\begin{equation} \label{domination}
\theta_N(e)\le c(e)\qquad \forall e\in E_N,
\end{equation}
when we identify $E_N$ with the edges of $E$ such that $\ue\in[-N/2,N/2[^d$.
\end{lemma}
%

The strategy now is to use this result to find a sequence
$(\theta_N)$ which satisfies~(\ref{sup-alpha}).
Let $(\alpha^{(i)}(e))_{e\in E}$ be the weights obtained by increasing the
weight $\alpha$ by $\kappa$ on the edge $(0,e_i)$, and leaving the
other values unchanged
\[
\alpha^{(i)}(e)=\cases{
\alpha^{(i)}(e)=\alpha(e),\qquad $\mbox{if $e\neq(0,e_i)$,}$\vspace*{2pt}\cr
\alpha^{(i)}((0,e_i))=\alpha((0,e_i))+\kappa=\alpha_i+\kappa.&}
\]
Let us first note that for all $i=1, \ldots, 2d,$
%
\begin{equation}\label{mincut}
m\bigl(\bigl(\alpha^{(i)}\bigr)\bigr)\ge\kappa.
\end{equation}
Take $i=1,\ldots, d$: if $S$ contains the edge $(0,e_i)$, then
$\alpha^{(i)}(S)\ge\alpha^{(i)}_{(0,e_i)}\ge\kappa$. Otherwise,
for all $j=1, \ldots,d$, $j\neq i$, $S$ must intersect the paths
$(k e_j)_{k\in\N}$, $(-ke_j)_{k\in\N}$, $(0,e_i, (e_i+k
e_j)_{k\in\N})$, $(0,e_i, (e_i-k e_j)_{k\in\N})$. These
intersections are disjoints, and it gives two edges with weights
$(\alpha_j)$ and two edges with weights $(\alpha_{j+d})$.
Moreover, $S$ must intersect the paths $(k e_i)_{k\in\N}$,
$(-ke_i)_{k\in\N}$. It gives one edge with weight~$\alpha_i$ and
one with weight $\alpha_{i+d}$. Hence,
\[
\alpha^{(i)}(S)\ge2\Biggl(\sum_{j=1}^{2d} \alpha_j\Biggr)
-(\alpha_i+\alpha_{i+d})\ge\kappa.
\]
The same reasoning works for $i=d+1, \ldots, 2d$.

Let us now prove (\ref{sup-alpha}) for $i=1$; the same reasoning
works for all. We apply Lemma~\ref{L2-max-flow} with
$c(e)=\alpha^{(1)}(e)$. It gives for $N$ large enough a function
$\tilde\theta_N:E_N\mapsto\R_+$ which satisfies
\[
\dive(\tilde\theta_N)={m(\alpha^{(1)})\over N^d}\sum_{y\in
T_N}(\delta
_0-\delta_y),
\]
and $\tilde\theta_N(e)\le\alpha^{(1)}(e)$ and with bounded
$L_2$ norm.
It implies that
$\theta_N={p \over m(\alpha^{(1)})} \tilde\theta_N$ satisfies
\[
\dive(\theta_N)={p \over N^d}\sum_{y\in T_N}(\delta_0-\delta_y),
\]
and by (\ref{mincut}) that $\theta_N(e)\le{p\over\kappa} \tilde
\theta
_N(e)\le{p\over\kappa}\alpha^{(1)}(e)$ and that $\theta_N$
has a bounded $L_2$-norm.

Let $r,q$ be positive reals such that ${1\over r}+{1\over q}=1$
and $pq<\kappa$. Using H\"{o}lder inequality and Lemma 1 of \cite
{Sabot}, we get
\begin{eqnarray*}
\E^{(\alpha)}\biggl(\omega (0,e_1)^\kappa{\check\omega ^{\check\theta
_N}\over
\omega ^{\theta_N}} \biggr)
&\le&\E^{(\alpha)}( \omega (0,e_1)^{q\kappa}\omega ^{-q
\theta_N})^{1/q} \E^{(\alpha)}( \check\omega ^{r\check
\theta_N})^{1/r}
\\
&=& \E^{(\alpha)}(\omega (0,e_1)^{q\kappa} \omega ^{-q \theta_N})^{1/q}
\E^{(\check\alpha)}( \omega ^{r\check\theta_N})^{1/r}.
\end{eqnarray*}
We set $\alpha(x)=\sum_{\ue=x}\alpha(e)$ and
$\theta_N(x)=\sum_{\ue=x} \theta_N(e)$.
Observe that $\alpha(x)=\check\alpha(x)=\sum_{j=1}^{2d} \alpha_j$ for
all $x\in T_N$.
We set $\alpha_0=\sum_{j=1}^{2d} \alpha_j$. For any function $\xi
:E_N\mapsto\R$ we have
%
\begin{equation}\label{xi}
\E^{(\alpha)}(\omega ^\xi)
=
\prod_{x\in T_N} \biggl({\prod_{i=1}^{2d} \Gamma(\alpha_i+\xi
(x,x+e_i))\over
\Gamma(\sum_{i=1}^{2d} \alpha_i+\xi(x,x+e_i))}{\Gamma(\alpha
_0)\over
\prod_{i=1}^{2d} \Gamma(\alpha_i)}\biggr)
\end{equation}
if $\xi(x,x+e_i)>-\alpha_i$ for all $x\in T_N$, $i=1, \ldots, 2d$, and
$+\infty$, otherwise.
Indeed, using the explicit form of Dirichlet distribution (\ref
{simplex}) and the independence at
each site, we get for any $\xi:E_N\mapsto\R$,
\[
\E^{(\alpha)}(\omega ^\xi)
=\biggl({\Gamma(\alpha_0)\over\prod_{i=1}^{2d} \Gamma(\alpha
_i)}\biggr)^{\vert
T_N\vert}
\prod_{x\in T_N} \int\prod_{i=1}^{2d} x_i^{\alpha_i+\xi(x,x+e_i)-1}
\,dx_1\cdots \,dx_{2d-1}
\]
where the last integrals are on the simplex $\{(x_1, \ldots, x_{2d}),
x_i>0, \sum x_i =1\}$. These integrals are Dirichlet integrals which
are finite if and only if $\alpha_i+\xi(x,x+e_i)>0$ for all $x$ and
$i$. There explicit value [cf. (\ref{simplex})] gives formula (\ref{xi}).
A straightforward application of (\ref{xi}) gives
\begin{eqnarray*}
&&\E^{(\alpha)}(\omega (0,e_1)^{q\kappa} \omega ^{-q \theta_N})
\\
&&\qquad= \Biggl({\prod_{{e\in E_N\atop e\neq(0,e_1)}} \Gamma(\alpha(e)-q\theta
_N(e))\over
\prod_{{x\in T_N\atop x\neq0}}
\Gamma(\alpha_0-q\theta_N(x))}\Biggr)
\biggl({\Gamma(\alpha_1+q\kappa-q\theta_N((0,e_1)))\over\Gamma(\alpha
_0+q\kappa-q\theta_N(0))}\biggr)\\
&&\quad\qquad{}\times\biggl({\prod_{{x\in T_N}}\Gamma(\alpha_0)\over\prod_{{e\in E_N}}\Gamma
(\alpha(e))}\biggr).
\end{eqnarray*}
Observe that all the terms are well defined since $q\theta_N \le
{pq\over\kappa} \alpha^{(1)}$ and $qp<\kappa$.
We have the following inequalities:
\[
\alpha_1\biggl(1-{qp\over\kappa}\biggr)\le
\alpha_1+q\kappa-q\theta_N((0,e_1)) \le
\alpha_1+q\kappa
\]
and
\[
\alpha_0\biggl(1-{qp\over\kappa}\biggr)\le
\alpha_0+q\kappa-q\theta_N(0) \le
\alpha_0 +q\kappa,
\]
which imply that
\begin{eqnarray*}
&&\E^{(\alpha)}(\omega (0,e_1)^{q\kappa} \omega ^{-q \theta_N})^{1/q}
\\
&&\qquad\le
A_1 \Biggl({\prod_{{e\in E_N\atop e\neq(0,e_1)}} \Gamma(\alpha(e)-q\theta
_N(e))\over
\prod_{{x\in T_N\atop x\neq0}}
\Gamma(\alpha_0-q\theta_N(x))}\Biggr)^{1/q}
\Biggl({\prod_{{x\in T_N\atop x\neq0}}\Gamma(\alpha_0)\over\prod_{{e\in
E_N\atop e\neq(0,e_1)}}\Gamma(\alpha(e))}\Biggr)^{1/q},
\end{eqnarray*}
where
\[
A_1=\biggl({\Gamma(\alpha_0)\over\Gamma(\alpha_1)}{\sup_{s\in[\alpha
_1(1-{qp/\kappa}), \alpha_1+q\kappa]} \Gamma(s),
\over\inf_{s\in[\alpha_0(1-{qp/\kappa}), \alpha_0+q\kappa]}
\Gamma(s)}\biggr)^{1/q}.
\]
Similarly, we get
\begin{eqnarray*}
\E^{(\check\alpha)}( \omega ^{r\check
\theta_N})
&=&
\biggl({\prod_{{e\in E_N}} \Gamma(\check\alpha(e)+r\check\theta
_N(e))\over
\prod_{{x\in T_N}} \Gamma(\check\alpha(x)+r\check\theta
_N(x))}\biggr)\biggl({\prod
_{{x\in T_N}}\Gamma(\check\alpha(x))\over\prod_{{e\in
E_N}}\Gamma(\check\alpha(e))}\biggr)
\\
&=&
\biggl({\prod_{{e\in E_N}} \Gamma(\alpha(e)+r\theta_N(e))\over
\prod_{{x\in T_N}} \Gamma(\alpha_0+r\check\theta_N(x))}\biggr)\biggl({\prod
_{{x\in
T_N}}\Gamma(\alpha_0)\over\prod_{{e\in
E_N}}\Gamma(\alpha(e))}\biggr),
\end{eqnarray*}
where in the last line we used that $\check\alpha((x,y))=\alpha((y,x))$
and $\check\theta((x,y))=\theta((y,x))$ and that
$\check\alpha(x)=\sum_{\overline e=x} \alpha_e=\alpha(x)=\alpha
_0$ for
all $x$.
Note that $\check\theta(0)=\theta(0)-p$ and $\check\theta
(x)=\theta
(x)+{p\over N^d}$ for $x\neq0$, thanks to
(\ref{dive-theta}). We have the following inequalities:
\begin{eqnarray*}
\alpha_1&\le&\alpha((0,e_1))+r\theta_N((0,e_1))\le\alpha
_1(1+r)+r\kappa,
\\
\alpha_0&\le&\alpha(0)+r\check\theta_N(0)\le\alpha_0(1+r)+r\kappa.
\end{eqnarray*}
This gives that
\begin{eqnarray*}
&&\E^{(\check\alpha)}( \omega ^{r\check
\theta_N})^{1/ r}\\
&&\qquad\le
A_2\Biggl({\prod_{{e\in E_N\atop e\neq(0,e_1)}} \Gamma(\alpha(e)+r\theta
_N(e))\over
\prod_{{x\in T_N\atop x\neq0}} \Gamma(\alpha_0+r\theta
_N(x)+{pr/
N^d})}\Biggr)^{1/r}\Biggl({\prod_{{x\in T_N\atop x\neq0}}\Gamma(\alpha_0)\over
\prod_{{e\in E_N\atop e\neq(0,e_1)}}\Gamma(\alpha(e))}\Biggr)^{1/r},
\end{eqnarray*}
where
\[
A_2=\biggl({\Gamma(\alpha_0)\over\Gamma(\alpha_1)}{\sup_{s\in[\alpha_1,
\alpha_1(1+r)+r\kappa]} \Gamma(s)
\over\inf_{s\in[\alpha_0, \alpha_0(1+r)+r\kappa]} \Gamma(s)}\biggr)^{1/r}.
\]
Combining these inequalities it gives
\begin{eqnarray*}
&&\E^{(\alpha)}\biggl(\omega (0,e_1)^\kappa{\check\omega ^{\check
\theta_N}\over
\omega ^{\theta_N}}\biggr)
\\
&&\qquad\le {A_1 A_2}\exp\Biggl(\mathop{\sum_{e\in
E_N}}_{e\neq(0,e_1)}\nu(\alpha(e), \theta_N(e))-\mathop{\sum_{x\in
T_N}}_{ x\neq0}\tilde\nu(\alpha_0,\theta_N(x))\Biggr),
\end{eqnarray*}
where
\[
\nu(\alpha, u)={1\over r} \ln\Gamma(\alpha+r u)+{1\over
q}\ln\Gamma(\alpha-q u)-\ln\Gamma(\alpha)
\]
and
\[
\tilde\nu(\alpha,u)={1\over r} \ln\Gamma\biggl(\alpha+r u+{pr\over
N^d}\biggr)+{1\over
q}\ln\Gamma(\alpha-q u)-\ln\Gamma(\alpha).
\]
Let $\underline\alpha=\min\alpha_i$, $\overline\alpha=\max\alpha_i$.
By Taylor's inequality and since $\underline\alpha\le\alpha(e)\le
\overline\alpha$ for all $e\in E_N$,
$q\theta_N(e)\le{qp\over\kappa} \alpha(e)$ for all $e\neq(0,e_1)$
and $qp<\kappa$, we can find a constant $c>0$
such that for all $e\neq(0,e_1)$,
\[
|\nu(\alpha(e),\theta(e))|\le
c\theta(e)^2
\]
and for all $x\neq0$,
\[
|\tilde\nu(\alpha_0,\theta(x))|\le
c\biggl(\theta(x)^2+{p\over N^d}\biggr).
\]
Hence, we get a positive
constant $C>0$
independent of $N>N_0$ such that
\[
\E^{(\alpha)}\biggl(\omega (0,e_1)^\kappa{\check\omega ^{\check\theta
_N}\over
\omega ^{\theta_N}}\biggr)\le\exp\biggl( C\biggl(\sum_{e\in E_N}
\theta_N(e)^2+\sum_{x\in T_N} \theta_N(x)^2\biggr)\biggr).
\]
Thus (\ref{sup-alpha}) is true, and this proves Lemma
\ref{main-lemma}.
\end{pf*}

\section{\texorpdfstring{Proof of Lemma \protect\ref{L2-max-flow}}{Proof of Lemma 2}}\label{sec4}

The strategy is to apply the max-flow min-cut theorem (cf. \cite
{Lyons-Peres}, Section 3.1 or
\cite{Ford-Fulkerson}) to an appropriate choice of capacities on the graph~$G_N$.
We first need a generalized version of the max-flow min-cut
theorem.
\begin{proposition}\label{max-flow}
Let $G=(V,E)$ be a finite directed graph. Let $(c(e))_{e\in E}$ be
a set of nonnegative reals (called capacities). Let $x_0$ be a
vertex and $(p_x)_{x\in V}$ be a set of nonnegative reals. There
exists a nonnegative function $\theta:E\mapsto\R_+$ such that
%
\begin{eqnarray} \label{dive-theta1}
&\displaystyle\dive(\theta)=\sum_{x\in V} p_x(\delta_{x_0}-\delta_x)&,
\\
\label{capacity}
&\displaystyle\forall e\in E, \qquad \theta(e)\le c(e),&
\end{eqnarray}
if and only if for all subset $K\subset V$ containing $x_0$ we
have
%
\begin{equation}\label{cutset-inequality}
c(\partial_+K)\ge\sum_{x\in K^c} p_x,
\end{equation}
where $\partial_+K=\{e\in E, \ue\in K, \oe\in K^c\}$ and
$c(\partial_+K)=\sum_{e\in\partial_+K}c(e)$.
The same is true if we restrict condition (\ref{cutset-inequality})
to the subsets $K$ such that any
$y\in K$ can be reached from 0 following a directed path in $K$.
\end{proposition}
\begin{pf}
If $\theta$ satisfies (\ref{dive-theta1}),
then
\[
\sum_{e, \ue\in K, \oe\in K^c} \theta(e)- \sum_{e, \oe\in K, \ue
\in
K^c} \theta(e)
=
\sum_{x\in K} \dive(\theta)(x)=\sum_{x\in K^c} p_x.
\]
It implies (\ref{cutset-inequality}) by (\ref{capacity}) and positivity
of $\theta$.

The reversed implication is an easy consequence of the classical
max-flow min-cut theorem
on finite directed graphs (\cite{Lyons-Peres}, Section 3.1 or \cite
{Ford-Fulkerson}).
Suppose now that~(c) satisfies (\ref{cutset-inequality}). Consider
the new graph
$\tilde G=(V\cup{\delta}, \tilde E)$ defined by
\[
\tilde E=E\cup\{(x,\delta), x\in V\}.
\]
We consider the capacities $(\tilde c(e))_{e\in\tilde E}$ defined by
$c(e)=\tilde c(e)$ for $e\in E$ and $c((x,\delta))= p_x$. The strategy
is to apply the max-flow
min-cut theorem with capacities $\tilde c$ and with source $x_0$ and
sink $\delta$. Any minimal cut-set between
$x_0$ and $\delta$ in the graph $\tilde G$ is of the form
$
\partial^{\tilde G}_+ K
$
where $K\subset V$ is a subset containing $x_0$ but not $\delta$ and
such that any point $y\in K$ can be reached from $x_0$
following a directed path in~$K$. Observe that
\[
\tilde c(\partial^{\tilde G}_+K)= c(\partial^{G}_+K)+ \sum_{x\in K} p_x.
\]
Hence, (\ref{cutset-inequality}) implies
\[
\tilde c( \partial^{\tilde G}_+ K)\ge\sum_{x\in V} p_x.
\]
Thus the max-flow min-cut theorem gives a flow $\tilde\theta$ on
$\tilde G$
between $x_0$ and $\delta$ with strength $\sum_{x\in V} p_x$ and such
that $\tilde\theta\le\tilde c$.
This necessarily implies that $\tilde\theta((x,\delta))=p_x$. The
function $\theta$ obtained by restriction of
$\tilde\theta$ to $E$ satisfies (\ref{capacity}) and (\ref{dive-theta1}).
\end{pf}

\begin{lemma}\label{flowL2} Let $d\ge3$.
There exists a positive constant $C_2>0$, such that for all $N>1$,
and all $x,y$ in $T_N$ there exists a unit flow $\theta$ from $x$
to $y$ (i.e., $\theta:E_N\to\R_+$ and
$\dive(\theta)=\delta_x-\delta_y$) such that
for all $z\in T_N$,
%
\begin{equation}\label{theta-small}
\theta(z)=\sum_{\ue=z}\theta(e)\le1\wedge\bigl(C_2 \bigl(
d(x,z)^{-(d-1)}+d(y,z)^{-(d-1)}\bigr)\bigr).
\end{equation}
%
\end{lemma}
\begin{pf}
By translation and symmetry, we can consider only the case where
$x=0$ and $y\in[N/2, N[^d$ when $T_N$ is identified with
$[0,N[^d$. We construct a flow on $G_{\Z^d}$ supported by the set
\[
D_y=[0,y_1]\times\cdots\times[0,y_d]
\]
as an integral of
sufficiently dispersed path flows. It thus induces by projection a
flow on $T_N$ with the same $L_2$ norm. Let us give some
definitions. A sequence $\sigma=(x_0, \ldots, x_n)$ is a path from
$x$ to $y$ in $\Z^d$ if $x_0=x$, $x_n=y$ and $\|x_{i+1}-x_i\|_1=1$
for all $i=1, \ldots, n$. We say that $\sigma$ is a positive path
if moreover $x_{i+1}-x_i\in\{e_1, \ldots, e_d\}$ for all $i=1,
\ldots, n$. To any path from $x$ to $y$ we can associate the unit
flow from $x$ to $y$ defined by
\[
\theta_\sigma=\sum_{i=1}^n \indic_{(x_{i-1},x_i)}.
\]
For $u\in\R_+$, we define $C_u$ by
\[
C_u=\Biggl\{z=(z_1, \ldots, z_d)\in\R_+^d, \sum_{i=1}^d z_i=u\Biggr\}.
\]
Clearly if $y\in\N^d$ and if $\sigma=(x_0=0, \ldots, x_n=y)$ is a
positive path from 0 to $y$, then $n=\|y\|_1$ and $x_k\in C_k$ for
all $k=0, \ldots, \|y\|$.

Set
\[
\Delta_y=D_y\cap\Biggl\{u=(u_1, \ldots, u_d)\in\R_+^d, \sum_{i=1}^d
u_i={\|y\|_1\over2}\Biggr\}.
\]
For $u\in\Delta_y$, let $L_u$ be the union of segments
\[
L_u=[0,u]\cup[u,y].
\]
We can consider $L_u$ as the continuous path $l_u:[0,\|y\|]\mapsto
D_y$ from 0 to $y$ defined by
\[
\{l_u(t)\}=L_u\cap C_t.
\]
Observe that $u\in D_y$ implies that $l_u(t)$ is nondecreasing on each
coordinate.
There is a canonical way to associate with $l_u$ a discrete
positive path $\sigma_u$ from 0 to $y$ such that
for all $k=0, \ldots, \|y\|$,
%
\begin{equation}\label{distance}
\|l_u(k)-\sigma_u(k)\|\le2d.
\end{equation}
Indeed, let $\tilde l_u(t)$ be defined by taking the integer part
of each coordinate of $l_u(t)$. At jump times of $\tilde l_u(t)$
the coordinates increase at most by 1. We define $\sigma_u(k)$ as
the positive path which follows the successive jumps of $\tilde
l_u(t)$: if at a time $t$ there are jumps at several coordinates,
we choose to increase first the coordinate on $e_1$, then on
$e_2$\ldots We have by construction $k-d\le\|\tilde l_u(k)\|\le k$, hence
$\tilde l_u(k) \in\{\sigma_u(k-d), \ldots, \sigma_u(k)\}$,
so $\|\sigma_u(k)-\tilde l_u(k)\|\le d$. Since $\|l_u(k)-\tilde
l_u(k)\|\le d$ it gives~(\ref{distance}). We then define
\[
\theta_u=\theta_{\sigma_u},
\]
and
\[
\theta={1\over|\Delta_z |}\int_{\Delta_z} \theta_u \,du
\]
(where $\vert\Delta_z\vert=\int_{\Delta_z} \,du$), which is a unit flow
from 0 to
$y$. Clearly, $\theta(z)\le1$ for all $z\in T_N$.
For $k=0, \ldots, \|y\|_1$ and $z\in\hhh_k$, we have
\[
\theta(z)\le{1\over\vert\Delta_z\vert}\int_{\Delta_z}
\indic_{\|l_u(k)-z\|\le2d} \,du.
\]
Hence, we have for $k$ such that $1<k\le{\|y\|\over2}$,
\[
\theta(z)\le{1\over\vert\Delta_z\vert}\int_{\Delta_z}
\indic_{\|u-z{\|y\|/(2k)}\|\le{d\|y\|/ k}} \,du.
\]
Since $y\in[N/2,N[^d$, there is a constant $C_2>0$ such that
\[
\theta(z)\le C_2 k^{-(d-1)}.
\]
Similarly, if $\| y\| /2 \le k< \|y\|$,
\[
\theta(z)\le C_2 (\|y\|-k)^{-(d-1)}.
\]
Moreover $\theta(z)$ is null on the complement of $D_y$. By projection
on $G_N$ it gives a function on $E_N$
with the right properties. This proves Lemma~\ref{flowL2}.
\end{pf}

We are ready to prove Lemma~\ref{L2-max-flow}. Let $(c(e))$ be such
that $0<C'<c(e)<C''<\infty$. For all $y\in
T_N$ we denote by $\theta_{0,y}$ a unit flow from 0 to $y$
satisfying the conditions of Lemma~\ref{flowL2}. We set
\[
\tilde\theta_N={m(c)\over N^d}\sum_{y\in T_N} \theta_{0,y}.
\]
The strategy is to apply proposition~\ref{max-flow} to a set of capacities
constructed from $\tilde\theta_N$ and $c$.
Clearly,
%
\begin{eqnarray}\label{dive}
\dive(\tilde\theta_N)={m(c)\over N^d}\sum_{y\in T_N}(\delta
_0-\delta_y),
\end{eqnarray}
and by simple computation we get that
%
\begin{equation}\label{theta-petit}
\tilde\theta_N(z)\le C_2 m(c)\biggl(1\wedge\bigl(d(0,z)^{-(d-1)}\bigr)+{d2^d\over
N^{d-1}}\biggr).
\end{equation}
This implies that
\begin{eqnarray*}
\sum_{z\in T_N}\tilde\theta_N^2(z)&\le&
2C_2^2 m(c)^2\sum_{z\in T_N} \biggl( 1\wedge\bigl(d(0,z)^{-2(d-1)}\bigr) +
{(d2^d)^2\over N^{2(d-1)}}\biggr)
\\
&\le&
2C_2^2 m(c)^2\biggl((d2^d)^2 N^{-d+2} + \sum_{z\in\Z^d} 1\wedge
\bigl(d(0,z)^{-2(d-1)}\bigr)\biggr).
\end{eqnarray*}
Considering that the number of points $z$ at distance $k$ from 0 is
smaller than $2d(2k+1)^{d-1}$, we get that
\[
\sum_{z\in T_N}\tilde\theta_N^2(z) \le
2C_2^2 m(c)^2\Biggl( (d2^d)^2 N^{-d+2} + 1+\sum_{k=1}^\infty k^{-2d+2}
(2d)(2k+1)^{d-1} \Biggr).
\]
Hence,
\[
\sum_{z\in T_N}\tilde\theta_N^2(z)\le2C_2^2m^2(c) (d2^d)\Biggl(1+\sum
_{k=1}^\infty k^{-(d-1)}\Biggr)+2C_2^2m^2(c)(d2^d)^2N^{-(d-2)},
\]
and there is a constant $C_3>0$ depending solely on $C', C'', d$ such that
\[
\sum_{z\in T_N}\tilde\theta_N^2(z)\le C_3,\qquad \sum_{e\in E_N}\tilde
\theta
_N^2(e)\le C_3.
\]
By (\ref{dive}) we know that for all $K\subset T_N$ containing $0$
we have
\[
\sum_{e\in E_N, \ue\in K, \oe\in K^c} \tilde\theta_N(e)-\sum
_{e\in
E_N, \oe\in K, \ue\in K^c} \tilde\theta_N(e)=m(c){\vert
K^c\vert\over N^d},
\]
hence,
%
\begin{equation}\label{cut-set}
\tilde\theta_N(\partial_+ K)\ge m(c){\vert K^c\vert\over N^d}.
\end{equation}
The strategy is to modify $\tilde\theta_N$ locally around 0 in order to
make it lower or equal to~$c$ but large enough to be able to apply
Proposition~\ref{max-flow}. Let us fix some notations. For a
positive integer $r$, $B_E(x_0,r)$ denotes the set of edges
\[
B_E(x_0,r)=\{e\in E, \ue\in B(x_0,r), \oe\in B(x_0,r)\}
\]
and
\[
\underline B_E(x_0,r)=\{e\in E, \ue\in B(x_0,r) \}.
\]
By (\ref{theta-petit}), there exist some positive integer $\eta_0$ and
$\tilde N_0$, such that for all
$N\ge\tilde N_0$ and $e\notin B_E(0,\eta_0)$, we have
%
\begin{equation}
\label{1.1}
\vert\tilde\theta_N(e)\vert\le{C'\over2}.
\end{equation}
Choose now $\eta_1>\eta_0$ such that
%
\begin{equation}
\label{1.2}
\eta_1-\eta_0\ge4{m(c)\over C'} +2 .
\end{equation}
Finally we can find an integer $N_0\ge\max(\tilde N_0, 2\eta_1)$ large
enough to satisfy
%
\begin{equation}
\label{1.3}
N_0^d\ge m(c) { \vert B(0,\eta_1)\vert\over C'}.
\end{equation}
We consider $(\tilde c_N(e))_{e\in E_N}$ defined by
\[
\cases{
\tilde c_N(e)= c(e), &\quad $\mbox{if $\ue$ or $\oe\in
B(0, \eta_1)$},$\vspace*{2pt}\cr
\tilde c_N(e)=\tilde\theta_N(e),&\quad $\mbox{otherwise}.$}
\]
Note that, thanks to (\ref{1.1}), for all $e\in E_N$, $\tilde
c_N(e)\le c(e)$
when we identify $E_N$ with the edges of
$E$, which starts in $[-N/2,N/2[^d$.
In the rest of the proof we prove that for all $N\ge N_0$ and for all
$K\subset T_N$ that contains 0 and which are such that any $y\in K$ can
be reached from
0 following a directed path in $K$, we have
%
\begin{equation}\label{K}
\tilde c_N(\partial_+K)\ge m(c){\vert K^c\vert\over N^d}.
\end{equation}
By application of Proposition~\ref{max-flow}, it would give a flow
$\theta_N$, which satisfies (\ref{flow-equation}) and~(\ref{domination}), and with a bounded $L_2$ norm, indeed,
\[
\sum_{e\in E_N} \theta_N(e)^2 \le C_3+ \vert B_E(0,\eta_1)\vert(C'')^2.
\]
We only need to check inequality (\ref{K}) for $K$ such that $K^c$ has a
unique connected component. Indeed, if $K^c$ has several connected
components, say $R_1, \ldots, R_k,$
then
\[
\partial_-R_i=\{e\in E_N, \oe\in R_i, \ue\in R_i^c\}=\{e\in E_N, \oe
\in
R_i, \ue\in K\}.
\]
Hence, $\partial_+K$ is the disjoint union of
\[
\partial_+K=\bigsqcup_{i=1}^k \partial_- R_i.
\]
Hence if we can prove (\ref{K}) for $K_i=R_i^c$, we can prove it for $K$.
Thus we assume, moreover, that $K^c$ has a unique connected component
in the graph $G_N$.
There are four different cases:
\begin{itemize}
\item If $K\subset B(0, \eta_1)$, then
\[
\tilde c_N(\partial_+K)=c(\partial_+K).
\]
Moreover, viewed on $\Z^d$ (when $T_N$ is identified with $[-N/2,N/2[$)
$\partial_+K$ is a cut-set separating
0 from $\infty$ (indeed, $N\ge N_0\ge2\eta_1$), thus
\[
c(\partial_+K)\ge m(c)\ge m(c){\vert K^c\vert\over
N^d}.
\]
\item If $B(0, \eta_0)\subset K$, by (\ref{cut-set}) and (\ref
{1.1}), then
\[
\tilde c_N(\partial_+K)\ge\tilde\theta_N (\partial_+K)\ge
m(c){\vert K^c\vert\over N^d}.
\]
\item
If $K^c\subset B(0, \eta_1)$, then by (\ref{1.3}),
\[
{\vert K^c\vert\over N^d}\le{\vert B(0,\eta_1)\vert\over N_0^d}\le
{C'\over m(c)},
\]
hence,
\[
\tilde c_N(\partial_+K)=c(\partial_+K)\ge C'\ge m(c) {\vert
K^c\vert\over N^d},
\]
since $\partial K^c\neq\varnothing$.
\item
Otherwise $K$ contains at least one
point $x_1$ in $B(0, \eta_1)^c$, and
$K^c$ contains at least one point $y_0$ in $B(0,\eta_0)$
and one point $y_1$ in $B(0,\eta_1)^c$. Hence there is a path
between $y_0$ and $y_1$ in $K^c$ and a directed path between 0 and
$x_1$ in $K$. Let $S(0,i)$ denote the sphere with center 0 and radius
$i$ for the shortest path distance
in $G_N$. It implies that we can find a sequence $z_{\eta_0},
\ldots, z_{\eta_1}$ such that $z_i\in K\cap S(0,i)$ and a sequence
$z'_{\eta_0}, \ldots, z'_{\eta_1}$ such that $z'_i\in K^c\cap
S(0,i)$. Since there is a directed path in $S(0,i)\cup S(0,i-1)$
between $z_i$
and $z'_i$, and a directed path in $K$ between 0 and $z_i$,
it implies that there exists at least $\lfloor\demi(\eta_1-\eta
_0)\rfloor$
different edges in $\partial_+K \cap B_E(0,\eta_1)$. Hence
\[
\tilde c_N(\partial_+K)\ge\bigl\lfloor\tfrac{1}{2}(\eta_1-\eta_0)\bigr\rfloor
C'\ge m((c)).
\]
\end{itemize}
This concludes the proof of (\ref{K}) and of the lemma.

\section{\texorpdfstring{Proof of Theorem \protect\ref{main-result}(ii) and Theorem \protect\ref{ballisticity}}
{Proof of Theorem 1(ii) and Theorem 2}}\label{sec5}

These results are based on classical results on ergodic stationary
sequence; cf.~\cite{Durrett}, pages 343--344.
Let us start with the following lemma.
\begin{lemma}\label{ergodic}
Suppose that there exists an invariant probability measure $\Q
^{(\alpha
)}$, absolutely continuous with respect to $\P^{(\alpha)}$
and invariant for $R$. Then $\Q^{(\alpha)}$ is equivalent to $\P
^{(\alpha)}$ and the Markov chain $(\overline w_n)$
with generator $R$, and the initial law $\Q^{(\alpha)}$ is stationary
and ergodic. Let $(\Delta_i)_{i\ge1}$ be the sequence
\[
\Delta_i=X_{i}-X_{i-1}.
\]
Under the invariant annealed measure $\Q^{(\alpha)}_0(\cdot)=\Q
^{(\alpha
)}(P_0^\omega (\cdot))$, the sequence
$(\Delta_i)$ is stationary and ergodic.
\end{lemma}
\begin{pf}
The first assertion on $\Q^{(\alpha)}$ is classical and proved, for
example, in~\cite{Sznitman-ten}, Theorem 1.2.
Since $\Q^{(\alpha)}$ is an invariant probability measure for
$\overline\omega _n$, it is clear that
$(\Delta_i)$ is stationary.
The ergodicity of $(\Delta_i)$ is a consequence of the ergodicity of
$(\omega _n)$. Indeed, since the environment is i.i.d. and not deterministic,
there exists a measurable function $f:\W\times\W\mapsto\Z$ such that
a.s. $\Delta_i=f(\omega _{i-1},\omega _i)$ (indeed, for $\P
^{(\alpha)}$ almost
all $\omega $,
$\tau_x (\omega )=\omega $ if and only if $x=0$, which means that
the increment
$\Delta_i$ is almost surely uniquely determined by the observation of
$\omega _{i-1}$ and $\omega _i$). This implies the ergodicity of the sequence~$(\Delta_i)$.
\end{pf}
\begin{pf*}{Proof of Theorem \protect\ref{main-result}(ii)} Suppose that there
exists an
invariant probability measure $\Q^{(\alpha)}$, absolutely continuous
with respect to $\P^{(\alpha)}$
and invariant for $R$.
Since $(X_n)$ is $\P^{(\alpha)}_0$ a.s. (hence, $\Q^{(\alpha)}_0$ a.s.)
transient (\cite{Sabot}, Theorem~1),
it implies that
\[
E^{\Q^{(\alpha)}}\bigl(P_0^{\omega }(H_0^+=\infty)\bigr)>0,
\]
where $H_0^+$ is the first positive return time of $X_n$ to 0.
Let $R_n$ be the number of points visited by $(X_k)$ at time $n-1$
\[
R_n=\vert\{X_k, k=0, \ldots, n-1\}\vert.
\]
Theorem 6.3.1 of~\cite{Durrett} and Lemma~\ref{ergodic} tell that
%
\begin{equation}\label{Rn}
\P^{(\alpha)}_0 \mbox{ a.s.,}\qquad {R_n\over n}\to E^{\Q^{(\alpha
)}}\bigl(P_0^{\omega
}(H_0^+=\infty)\bigr)>0.
\end{equation}
Let $i_0\in\{1,\ldots,d\}$ be a direction which maximizes $\alpha
_i+\alpha_{i+d}$. Theorem 3 of~\cite{Tournier} tells that if $\kappa\le1$, then the expected exit
time under $\P^{(\alpha)}_0$ of the finite subset
$\{0,e_{i_0}\}$ or $\{0,-e_{i_0}\}$ is infinite. By independence of the
environment under $\P^{(\alpha)}$, we can easily get that
${R_n\over n}\to0$, $\P^{(\alpha)}_0$ a.s. This contradicts (\ref{Rn}).
\end{pf*}
\begin{pf*}{Proof of Theorem \protect\ref{ballisticity}}
(i) is Proposition 12 of~\cite{Tournier}.
Under the annealed invariant law $\Q^{(\alpha)}_0$, $(\Delta_k)$ is a
stationary ergodic sequence with values in $\Z^d$ (hence for any $i\in
\{1, \ldots, 2d\}$, $\Delta_k\cdot e_i$ is also a stationary ergodic
sequence with values in $\Z$). Birkhoff's ergodic Theorem (\cite
{Durrett}, page 337)
gives for free the law of large number
\[
\P_0^{(\alpha)}\mbox{ a.s.,}\qquad {X_n\over n}\to E^{\Q^{(\alpha
)}}(E^{\omega }_0(X_1)).
\]
If $d_\alpha\cdot e_i=0$ then by symmetry of the law of the environment
it implies that $E^{\Q^{(\alpha)}}(E^{\omega }_0(X_1))\cdot e_i=0$,
hence by Theorem 6.3.2 of~\cite{Durrett}, we have that $X_n\cdot e_i=0$
infinitely often. By Lemma 4 of~\cite{Zerner-Merkl} it implies (ii) and
the last assertion of (iii).

For $l\in\R^d$, we set $A_l=\{X_n\cdot l\to\infty\}$. If $l\neq0$
and if $\P^{(\alpha)}_0(A_l)>0$, then
Kalikow 0--1 law (\cite{Kalikow},~\cite{Zerner-Merkl}, Proposition 3)
tells that $\P^{(\alpha)}_0(A_l\cup A_{-l})=1$. Suppose now that
$d_\alpha\cdot e_i>0$ for an integer $i$ in $\{1, \ldots, 2d\}$. In
\cite{Sabot-Tournier} we proved that
$\P^{(\alpha)}_0(A_{e_i})>0$, this implies that $X_n\cdot e_i$ visits 0
a finite number of times $\Q^{(\alpha)}_0$ a.s. By Theorem 6.3.2
of~\cite{Durrett} it implies that
\[
E^{\Q^{(\alpha)}}(E_0^{\omega }(X_1))\cdot e_i \neq0.
\]
Moreover, we know that
\[
\P_0^{(\alpha)}\mbox{ a.s.,}\qquad {X_n\over n}\to E^{\Q^{(\alpha
)}}(E^{\omega }_0(X_1)).
\]
Hence, $\P^{(\alpha)}(A_{\pm e_i})=1$, where $\pm$ corresponds to the
sign of $E^{\Q^{(\alpha)}}(E^{\omega }_0(X_1))\cdot e_i$.
Since we know that $\P^{(\alpha)}_0(A_{e_i})>0$, it implies that
\[
E^{\Q^{(\alpha)}}(E_0^{\omega }(X_1))\cdot e_i > 0.
\]
\upqed\end{pf*}


%

\printaddresses

\end{document}